\documentclass{article}

\usepackage{graphicx,xcolor,tikz,float}
\usepackage{amsthm,amsmath,amssymb}
\usepackage{comment,hyperref}

\newtheorem{theorem}{Theorem}[section]

\newtheorem{definition}[theorem]{Definition}
\newtheorem{corollary}{Corollary}[theorem]
\newtheorem*{claim}{Claim}

\newtheorem*{remark}{Remark}

\newtheorem{proposition}[theorem]{Proposition}

\title{A Comparison of Gauge Dimension and\\ Effective Dimension}
\author{Yiping Miao}
\date{}

\begin{document}

\maketitle

\begin{abstract}
We characterize the gauge profiles of $\mathcal{D}_s$, the set of reals with effective dimension $s$, and $\mathcal{D}_{\leq s}$, the set of reals with effective dimension $\leq s$. Let $W(s)$ be the set of reals that are $s$-well approximable. This gives us a separation between $\mathcal{D}_{\leq s}$ and $W(2/s)$ in terms of Hausdorff measure.
\end{abstract}

\section{Introduction}

A {\bf gauge function} $f:\mathbb{R}^+\rightarrow\mathbb{R}^+$ is a continuous, non-decreasing function with $\lim_{x\rightarrow 0+}f(x)=0$. For any $A\subseteq 2^\omega$, we define the $f$-measure of $A$ as follows

$$H^f(A)=\lim_{n\rightarrow \infty}\inf\left\{\sum_i f(2^{-|\sigma_i|}):\bigcup_i [\sigma_i]\supseteq A\text{ and }\forall i\ (|\sigma_i|\geq n)\right\}.$$

The gauge profile of $A$ is the set of gauge functions $\{f:H^f(A)>0\}$.

There are various ways to compare two gauge functions. We use the literal ordering here.

\begin{definition}
For any $f,g:\mathbb{R}^+\rightarrow \mathbb{R}^+$, $f(x)\leq^* g(x)$ if there is some $\delta>0$ such that for any $x\in (0,\delta)$, $f(x)\leq g(x)$.
\end{definition}

Now we introduce the main object we study in this paper. For any partial computable function $f:2^{<\omega}\rightarrow 2^{<\omega}$, we define $C_f(\sigma)$ to be the least length of string $\tau$ that can describe $\sigma$ via a program computing $f$, if it exists, otherwise $C_f(\sigma)=\infty$.
$$C_f(\sigma)=\inf\{|\tau|: \exists \tau\ f(\tau)=\sigma\}.$$

There are optimal partial computable functions $U$ in the sense that for any partial computable function $f$, there is some $D>0$ such that for any $\sigma$,
$$C_U(\sigma)\leq C_f(\sigma)+D,$$
Let $C_U(\sigma)$ be the {\bf plain Kolmogorov complexity} of $\sigma\in 2^{<\omega}$, which is well defined up to an additive constant. We will use $D$ as this $O(1)$ factor throughout the paper.

For $x\in 2^\omega$, let $dim(x)$ be the {\bf effective dimension} of $x$. One of the equivalent definitions is

$$dim(x)=\liminf_{n\rightarrow\infty}C(x\upharpoonright n)/n,$$

There are other ways to define effective Hausdorff dimensions, related works include \cite{cai1994hausdorff,lutz2000gales,reimann2005effective}. See also Downey and Hirschfeld's book\cite{downey2010algorithmic}.

\begin{definition}
For $0\leq s\leq 1$, let 
$$\mathcal{D}_s=\{x\in 2^\omega:dim(x)=s\},$$
$$\mathcal{D}_{\leq s}=\{x\in 2^\omega: dim(x)\leq s\}.$$
\end{definition}

Our main theorem characterizes the gauge profiles of $\mathcal{D}_s$ and $\mathcal{D}_{\leq s}$ under a mild restriction. A sufficient condition for this restriction is given in \cite{staiger2017exact}.

\vspace{2mm}
\noindent {\bf Theorem \ref{gauge_dim}.} If $x^{-1}f(x)$ is monotonically decreasing, and $0\leq s<1$, then the following are equivalent:
\begin{enumerate}
    \item $H^f(\mathcal{D}_s)>0$,
    \item $H^f(\mathcal{D}_{\leq s})>0$,
    \item $\forall r\in (s,1]\ f(x)\not\leq^* x^r$.
\end{enumerate}
\vspace{2mm}

This gives us a comparison between effective dimension and Diophantine approximation.

\begin{definition}[\cite{beresnevich2006measure}]
For $s\geq 2$, let $W(s)$ be the set of $s$-well approximable numbers, i.e.
\begin{equation*}
\begin{split}
W(s)=&\{x\in [0,1]: |x-\frac{p}{q}|<q^{-s}\\
&\text{ for infinitely many integer pairs }(p,q) \text{ with }q>0\}.
\end{split}
\end{equation*}
\end{definition}

Dirichlet's approximation theorem shows $W(2)=[0,1]$. A related notion is the irrationality exponent. The {\bf irrationality exponent} $a$ of $x\in [0,1]\setminus \mathbb{Q}$ is defined as 
$$a=\sup \{s:x\in W(s)\}.$$
The reals with infinite irrationality exponent are called Liouville numbers.

We do not choose to use this notion since we want to use the gauge profile of $W(s)$ obtained by Jarn\'{i}k \cite{jarnik1999diophantische}. See \cite{becher2018irrationality,calude2013liouville} for related work on irrationality exponent and effective Hausdorff dimensions.

$W(2/s)$ and $\mathcal{D}_{\leq s}$ are closely related. Calude and Staiger $\cite{calude2013liouville}$ showed that $W(2/s)\subseteq\mathcal{D}_{\leq s}$. Both sets have Hausdorff dimension $s$ \cite{jarnik1931simultanen,besicovitch1934sets,cai1994hausdorff,reimann2004computability}. Our characterization separates $W(2/s)$ and $\mathcal{D}_{\leq s}$ in measure.

\vspace{2mm}
\noindent {\bf Corollary \ref{separate}.} There is a gauge function $f$ such that $H^f(W(2/s))=0$ and $H^f(\mathcal{D}_{\leq s})>0$.
\vspace{2mm}

\section{Gauge Profile of the Set of Effective Dimension \texorpdfstring{$s$}{s}}

The following was first proved by Jin-Yi Cai and Juris Hartmanis for computable $s$ \cite{cai1994hausdorff}, and Jan Reimann for general $s$ \cite{reimann2004computability}. See \cite{ya1984coding,lutz2003dimensions,lutz2005effective,staiger1993kolmogorov} for related works.

\begin{theorem}[Cai, Hartmanis\cite{cai1994hausdorff}; Reimann\cite{reimann2004computability}]
$$dim_H(\mathcal{D}_s)=s.$$
\end{theorem}

We study the gauge profiles of $\mathcal{D}_s,\mathcal{D}_{\leq s}$.

2. $\Leftrightarrow$ 3. in the following theorem can be proved directly using mass transference principle \cite{beresnevich2006mass}, with the additional restriction that $x^{-1}f(x)\rightarrow\infty$ when $x\rightarrow 0$.

\begin{theorem}
\label{gauge_dim}
If $x^{-1}f(x)$ is monotonically decreasing, and $0\leq s<1$, then the following are equivalent:
\begin{enumerate}
    \item $H^f(\mathcal{D}_s)>0$,
    \item $H^f(\mathcal{D}_{\leq s})>0$,
    \item $\forall r\in (s,1]\ f(x)\not\leq^* x^r$.
\end{enumerate}
\end{theorem}

\begin{proof}
1. $\Rightarrow$ 2. is trivial. 2. $\Rightarrow$ 3. follows from the fact that $dim_H(\mathcal{D}_{\leq s})\leq s$. We prove 3. $\Rightarrow$ 1. below.

We work in $2^\omega$. From the definition of $H^f$ in $2^\omega$ we only need to consider the values of $f$ on $\{2^{-n}:n\in\omega\}$.

\begin{claim}
Under the assumptions of the theorem, 
$$\forall r\in (s,1]\forall N\in\omega\exists n\in\omega(n\geq N\wedge f(2^{-n})>2^{-r\cdot n}).$$
\end{claim}

\begin{proof}[Proof of the claim]
For any $r\in (s,1]$, fix some $\varepsilon>0$ such that $r-\varepsilon>s$. For any $N\in\omega$ such that $N\geq \varepsilon^{-1}$, there is some $x\in (0,2^{-N}]$ such that $f(x)>x^{r-\varepsilon}$. Let $n$ be the least $k\in\omega$ such that $2^{-k}\leq x$, then 
$$\frac{f(2^{-n})}{2^{-n}}\geq \frac{f(x)}{x}.$$
So $x^{-\varepsilon}\geq (2^{-N})^{-\varepsilon}=2^{\varepsilon\cdot N}\geq 2$.
$$f(2^{-n})\geq f(x)\cdot\frac{2^{-n}}{x}>x^{r-\varepsilon}\cdot \frac{2^{-n}}{x}\geq x^r\cdot\frac{2^{-n+1}}{x}>x^r.$$
The last inequality follows from $x<2^{-n+1}$ as $n$ is taken as the least.
\end{proof}

Fix a descending sequence $r_n\rightarrow s$, then for any $\varepsilon>0$, there is $x\in (0,\varepsilon)$ such that $f(x)> x^{r_n}$.

Define a sequence of integers $\{l_n\}$ recursively, let $l_0=0$, $r_{-1}=1$, $l_{n+1}$ be the least $l>2^n+l_n+\frac{2}{r_n-r_{n+1}}$ such that $\ulcorner r_n\cdot l\urcorner\geq \ulcorner r_{n-1}\cdot l_n\urcorner$ and $f(2^{-l})>2^{-l\cdot r_n}$. Let $r_n^*=\frac{\ulcorner r_n\cdot l_{n+1}\urcorner}{l_{n+1}}$. Then $r_n^*$ decreases, $r_n^*\rightarrow s$ and $f(2^{-l_{n+1}})>2^{-l_{n+1}\cdot r_n^*}$. Let $r^*_{-1}=1$.

We build a perfect tree $T$ so that the ratio of splits and non-splits below $l_{n+1}$ is $r_n^*$, and always splits at the lowest level possible.
$$T:=\left\{\sigma\in 2^{<\omega}: \forall n\forall l\ [r_n^*\cdot l_{n+1}+(1-r_{n-1}^*)\cdot l_n\leq l<l_{n+1}\rightarrow \sigma(l)=0] \right\}$$

 
\tikzset{
pattern size/.store in=\mcSize, 
pattern size = 5pt,
pattern thickness/.store in=\mcThickness, 
pattern thickness = 0.3pt,
pattern radius/.store in=\mcRadius, 
pattern radius = 1pt}
\makeatletter
\pgfutil@ifundefined{pgf@pattern@name@_5konwy22c}{
\pgfdeclarepatternformonly[\mcThickness,\mcSize]{_5konwy22c}
{\pgfqpoint{0pt}{0pt}}
{\pgfpoint{\mcSize+\mcThickness}{\mcSize+\mcThickness}}
{\pgfpoint{\mcSize}{\mcSize}}
{
\pgfsetcolor{\tikz@pattern@color}
\pgfsetlinewidth{\mcThickness}
\pgfpathmoveto{\pgfqpoint{0pt}{0pt}}
\pgfpathlineto{\pgfpoint{\mcSize+\mcThickness}{\mcSize+\mcThickness}}
\pgfusepath{stroke}
}}
\makeatother

 
\tikzset{
pattern size/.store in=\mcSize, 
pattern size = 5pt,
pattern thickness/.store in=\mcThickness, 
pattern thickness = 0.3pt,
pattern radius/.store in=\mcRadius, 
pattern radius = 1pt}
\makeatletter
\pgfutil@ifundefined{pgf@pattern@name@_ubtt6hapv}{
\pgfdeclarepatternformonly[\mcThickness,\mcSize]{_ubtt6hapv}
{\pgfqpoint{0pt}{0pt}}
{\pgfpoint{\mcSize+\mcThickness}{\mcSize+\mcThickness}}
{\pgfpoint{\mcSize}{\mcSize}}
{
\pgfsetcolor{\tikz@pattern@color}
\pgfsetlinewidth{\mcThickness}
\pgfpathmoveto{\pgfqpoint{0pt}{0pt}}
\pgfpathlineto{\pgfpoint{\mcSize+\mcThickness}{\mcSize+\mcThickness}}
\pgfusepath{stroke}
}}
\makeatother
\tikzset{every picture/.style={line width=0.75pt}} 

\begin{figure}[H]
    \centering
\begin{tikzpicture}[x=0.75pt,y=0.75pt,yscale=-1,xscale=1]

\usetikzlibrary{patterns}
\draw  [color={rgb, 255:red, 0; green, 0; blue, 0 }  ,draw opacity=1 ][pattern=_5konwy22c,pattern size=6pt,pattern thickness=0.75pt,pattern radius=0pt, pattern color={rgb, 255:red, 0; green, 0; blue, 0}] (284,223) -- (248,183) -- (318,183) -- cycle ;
\draw   (248,143) -- (318,143) -- (318,183) -- (248,183) -- cycle ;
\draw  [pattern=_ubtt6hapv,pattern size=6pt,pattern thickness=0.75pt,pattern radius=0pt, pattern color={rgb, 255:red, 0; green, 0; blue, 0}] (329,103) -- (317,143) -- (248,143) -- (236,103) -- cycle ;
\draw   (236,63) -- (329,63) -- (329,103) -- (236,103) -- cycle ;

\draw (198,210.4) node [anchor=north west][inner sep=0.75pt]    {$l_{0}$};
\draw (198,132.4) node [anchor=north west][inner sep=0.75pt]    {$l_{1}$};
\draw (199,52.4) node [anchor=north west][inner sep=0.75pt]    {$l_{2}$};
\draw (185,175.4) node [anchor=north west][inner sep=0.75pt]    {$r_{0}^{*} \cdot l_{1}$};

\end{tikzpicture}

    \caption{Tree $T$}
    \label{fig:tree}
\end{figure}

\begin{claim}
\begin{enumerate}
    \item For any $x\in [T]$, $dim(x)\leq s$.
    \item \footnote{Ludwig Staiger's paper\cite{staiger2017exact} section 4.1 contains a similar construction and proves a similar version of this claim with $A=2^\omega$.}Let $\delta:2^\omega\rightarrow [T]$ be the canonical homeomorphism, for any $A\subseteq 2^\omega$, 
    $$\mu(A)>0\Rightarrow H^f(\delta(A))>0.$$
    \item For any $x\in 2^\omega$, $dim^T(\delta(x))=s\cdot dim^T(x)$.
\end{enumerate}
\end{claim}

\begin{proof}[Proof of the Claim.]
\begin{enumerate}
    \item To record information for $x\upharpoonright l_{n+1}$, we only need to know $r_n^*\cdot l_{n+1}$ many bits (other bits are $0$), and on which levels the tree starts/stops splitting. More precisely, we need to know $r_k^*\cdot l_{k+1}+(1-r_{k-1}^*)\cdot l_k$ and $l_{k+1}$, for $k=0,1,...,n$. These are $2(n+1)$ natural numbers, each of them bounded by $l_{n+1}$. Let $\{\sigma_i\}_{i<2(n+1)}$ be the binary representation of this sequence of natural numbers. $0^{|\sigma_0|-1}1\sigma_0 0^{|\sigma_1|-1}1\sigma_1...$ will code this sequence of natural numbers.
    $$dim(x)\leq \liminf_{n\rightarrow\infty}\frac{C(x\upharpoonright l_{n+1})}{l_{n+1}}\leq \liminf_{n\rightarrow\infty}\frac{r_n^*\cdot l_{n+1}+4(n+1)\cdot \log l_{n+1}+D}{l_{n+1}}.$$
    Since $l_{n+1}>2^n$, $n<\log l_{n+1}$,
    $$dim(x)\leq \liminf_{n\rightarrow\infty}\left[r_n^*+\frac{4(\log l_{n+1}+1)\cdot \log l_{n+1}+D}{l_{n+1}}\right]=s.$$
    \item Let $k_n$ be the number of levels below $n$ where $T$ splits. More precisely, $k_n$ the number of $k<n$ such that there are $\sigma_0,\sigma_1\in T$ with $\sigma_0(k)=0$, $\sigma_1(k)=1$. For $\sigma_i\in T$, let $\tau_i$ be the unique string such that $\delta([\tau_i])=[\sigma_i]$.
    $$H^f(\delta(A))=\liminf_{n\rightarrow\infty}\left\{\sum_i f(2^{-|\sigma_i|}):\bigcup_i [\sigma_i]\supseteq\delta(A)\wedge |\sigma_i|\geq n\right\}$$
    $$\geq\liminf_{n\rightarrow\infty}\left\{\sum_i 2^{-|\tau_i|}:\bigcup_i[\tau_i]\supseteq A\wedge |\tau_i|\geq k_n\right\}.$$
    Here $\geq$ follows from the fact $f(2^{-l_{n+1}})>2^{-l_{n+1}\cdot r_n^*}$ and $f(x)/x$ is monotonically decreasing. In particular, If $\sigma_i$ is on the level the tree splits, let $n$ be the largest such that $l_{n+1}\leq |\sigma_i|$, 
    $$\frac{f(2^{-|\sigma_i|})}{2^{-|\sigma_i|}}\geq \frac{f(2^{-l_{n+1}})}{2^{-l_{n+1}}},$$
    so 
    $$f(2^{-|\sigma_i|})\geq f(2^{-l_{n+1}})\cdot 2^{-(|\sigma_i|-l_{n+1})}>2^{-(l_{n+1}\cdot r_n^*+|\sigma_i|-l_{n+1})}=2^{-|\tau_i|}.$$ 
    If $\sigma_i$ is on the level the tree does not split, let $n$ be the smallest such that $l_{n+1}\geq |\sigma_i|$, then $f(2^{-|\sigma_i|})\geq f(2^{-l_{n+1}})>2^{-l_{n+1}\cdot r_n^*}=2^{-|\tau_i|}$.
    \item We know
    $$C^T(\delta(x)\upharpoonright n)\leq C^T(x\upharpoonright k_n)+\log n+D,$$
    $$C^T(x\upharpoonright k_n)\leq C^T(\delta(x)\upharpoonright n)+D.$$
    So
    $$dim^T(\delta(x))=\liminf_{n\rightarrow\infty}\frac{C^T(\delta(x)\upharpoonright n)}{n}\leq\liminf_{n\rightarrow \infty}\frac{C^T(x\upharpoonright k_n)}{k_n}\cdot \frac{k_n}{n}+\frac{\log n}{n}+\frac{D}{n}=s\cdot dim^T(x).$$
    If $s\not= 0$, then
    $$dim^T(x)=\liminf_{n\rightarrow\infty}\frac{C^T(x\upharpoonright k_n)}{k_n}\leq \liminf_{n\rightarrow\infty}\frac{C^T(\delta(x)\upharpoonright n)}{k_n}+\frac{D}{k_n}=dim^T(\delta(x))\cdot \frac{1}{s}.$$
\end{enumerate}
\end{proof}

\begin{remark}
$[T]$ in the construction has effective dimension $s$.
\end{remark}

Let RAND be the set of $1$-random relative to $T$ reals, then $H^f(\delta(\text{RAND}))>0$. For any $x\in\delta(\text{RAND})$, $dim(x)\leq s$ and $dim(x)\geq dim^T(x)=s$. So $\delta(\text{RAND})\subseteq\mathcal{D}_s$, and $H^f(\mathcal{D}_s)>0$.
\end{proof}

The theorem only deals with $0\leq s<1$. For $s=1$, the Lebesgue measure $\mu(\mathcal{D}_1)=1$ (resp. $\mathcal{D}_{\leq 1}$). For any $A$ with $\mu(A)>0$, and any gauge function $f$,
$$H^f(A)>0\text{ if and only if }\liminf_{x\rightarrow 0+}x^{-1}f(x)>0.$$
Note that $x^{-1}f(x)$ decreases monotonically implies $\liminf_{x\rightarrow 0+}x^{-1}f(x)>0$.

The tree constructed in the proof gives an example of a closed set with positive measure under some dimension (or gauge function). There is previous work with similar constructions, for example \cite{staiger2020incomputability}.

\vspace{3mm}

Now we turn to the value of $H^f(\mathcal{D}_s)$.
\begin{corollary}
Suppose $x^{-1}f(x)$ decreases monotonically, and $0\leq s<1$. If $H^f(\mathcal{D}_s)>0$, $H^f(\mathcal{D}_s)$ is not $\sigma$-finite (resp. $\mathcal{D}_{\leq s}$).
\end{corollary}

\begin{proof}
This follows from the fact that if $x^{-1}f(x)$ is decreasing monotonically and for any $r\in(s,1]$, $f$ is not dominated by $x^r$, then there exists $g\succ f$ such that $x^{-1}g(x)$ is decreasing monotonically and for any $r\in(s,1]$, $g$ is not dominated by $x^r$.
\end{proof}

\begin{corollary}
For $0<s<1$, $\mathcal{D}_{<s}:=\mathcal{D}_{\leq s}-\mathcal{D}_s$, if $x^{-1}f(x)$ decreases monotonically,
\begin{center}
$H^f(\mathcal{D}_{<s})=0\text{ if and only if }\forall r\in (0,s)$ $f\leq^* x^r$.    
\end{center}
Moreover, $\mathcal{D}_{<s}$ has dimension $s$ and $H^s(\mathcal{D}_{<s})=0$
\end{corollary}

\begin{proof}
$\mathcal{D}_{<s}=\bigcup_{r\in \mathbb{Q}\cap (0,s)}\mathcal{D}_s.$
\begin{align*}
  H^f(\mathcal{D}_{<s})=0 & \Leftrightarrow\forall r\in (0,s)\ H^f(\mathcal{D}_r)=0\\
  & \Leftrightarrow\forall r\in (0,s)\exists k\in (r,1]\ f(x)\leq^* x^k\\
  & \Leftrightarrow\forall r\in (0,s)\ f(x)\leq^* x^r.
\end{align*}
\end{proof}

\section{A Comparison with Diophantine Approximation}

\subsection{From \texorpdfstring{$2^\omega$}{2omega} to \texorpdfstring{$[0,1]$}{[0,1]}}

For the definition of $H^f$ in $\mathbb{R}^n$, see Rogers' book \cite{alma991043247829706532}. We provide a discussion here that extends our result to $[0,1]$.

Let $\pi:2^\omega\setminus \{x:x\text{ is eventually }1\}\rightarrow [0,1]$ be
$$\pi(x)=\sum_{i\in x}2^{-i}.$$
$\pi$ is a bijection.

\begin{definition}
An interval $I$ is called dyadic if it is of the form $[i\cdot 2^{-n}, (i+1)\cdot 2^{-n})$ for any $i,n\in\mathbb{N}$ and $I\subseteq [0,1]$.
\end{definition}

\begin{proposition}
Let $f$ be a gauge function, $A\subseteq 2^\omega$, then
$$H^f(A)=H^f(\pi(A)).$$
\end{proposition}

\begin{proof}
Since $A$ and $A\setminus \{x:x\text{ is eventually }1\}$ have the same $H^f$ measure, we can assume $A\subseteq 2^\omega\setminus \{x:x\text{ is eventually }1\}$.\\
$$H^f(A)=\lim_{n\rightarrow \infty}\inf\left\{\sum_i f(2^{-|\sigma_i|}):\bigcup_i [\sigma_i]\setminus \{x:x\text{ is eventually }1\}\supseteq A\text{ and }\forall i\ (|\sigma_i|\geq n)\right\}.$$
For any $B\subseteq [0,1]$,
$$H^f(B)=\lim_{n\rightarrow \infty}\inf\left\{\sum_i f(|I_i|):I_i\text{'s are dyadic intervals, }|I_i|\leq 2^{-n},\ \bigcup_i I_i\supseteq B\right\}.$$
(See \cite{alma991043247829706532} for covering using dyadic intervals.)

\vspace{0.3cm}

For any $\sigma\in 2^{<\omega}$, $\pi([\sigma])=[\sum_{i\in\sigma}2^{-i},2^{-|\sigma|}+\sum_{i\in\sigma}2^{-i})$. So 
$$H^f(A)=H^f(\pi(A)).$$
\end{proof}

Since the effective dimension of $x\in [0,1]$ is defined as the effective dimension of its binary representation, working in $2^\omega$ and $[0,1]$ are the same.

\subsection{A Comparison with Diophantine Approximation}
Diophantine approximation is closely related to effective dimensions. In particular, in the definition of $W(s)$, $|x-\frac{p}{q}|<q^{-s}$ gives us an algorithm to approximate $x$ using a rational number, with an error of $q^{-s}$.

Jarn\'{i}k \cite{jarnik1999diophantische} and independently Besicovitch \cite{besicovitch1934sets} proves that the Hausdorff dimension of $W(s)$ is $2/s$. Jarn\'{i}k obtains the gauge profile of $W(s)$ as well \cite{jarnik1931simultanen}.

\begin{theorem}[Jarn\'{i}k\cite{jarnik1931simultanen}]
For any gauge function $f$ such that $\lim_{x\rightarrow 0+}x^{-1}f(x)=\infty$ and $x^{-1}f(x)$ decreases monotonically,
$$H^f(W(s))=0\text{ if and only if }\sum_{q=1}^\infty qf(q^{-s})<\infty.$$
\end{theorem}

Calude and Staiger showed in $\cite{calude2013liouville}$ that $W(2/s)\subseteq\mathcal{D}_{\leq s}$. We present a sketch of their proof here as this indicates how closely related these two sets are.

\begin{theorem}[Calude, Staiger\cite{calude2013liouville}]
For $0<s\leq 1$, $W(2/s)\subseteq\mathcal{D}_{\leq s}$.
\end{theorem}

\begin{proof}
If $x\in W(2/s)$, take $(p,q)$ such that $|x-\frac{p}{q}|<q^{-2/s}$. For $k\leq \frac{2}{s}\log q$, $(\frac{p}{q}-q^{-2/s},\frac{p}{q}+q^{-2/s})$ is contained in at most three dyadic intervals of length $2^{-k}$. So $(p,q)$ reveals the first $k$ digits, up to three possibilities. Let $k$ be the largest natural number such that $k\leq \frac{2}{s}\log q$. $(p,q,i)$, where $i=0,1,2$ depending on which of the three dyadic intervals $x$ is in, will code the first $k$ bits of $x$. We use the standard way to code $(p,q)$. See for example \cite{shen2015around} for more details. Let $\sigma,\tau,\gamma$ be the binary representation of $p,q,\log p$, and $\gamma_2$ be $\gamma$ with every bit doubled, a pair $(p,q)$ can be coded as $\gamma_20 1\sigma\tau$.
$$C(x\upharpoonright k)\leq \log p+\log q+2\log\log p+D\leq (\frac{2}{s}\log q)\cdot s+2\log\log q+D.$$
So
$$\liminf_{n\rightarrow\infty} C(x\upharpoonright n)/n\leq s.$$
\end{proof}

It is clear that if $\exists r\in (s,1]\ f(x)\leq^* x^r$, then $\sum_{q=1}^\infty qf(q^{-2/s})<\infty$.

\begin{proposition}
For $0<s<1$, there is a gauge function $f$ satisfying
\begin{enumerate}
    \item $\lim_{x\rightarrow 0+}x^{-1}f(x)=\infty$,
    \item $x^{-1}f(x)$ decreases monotonically,
    \item $\sum_{q=1}^\infty qf(q^{-2/s})<\infty$,
    \item $\forall r\in (s,1]\ f(x)\not\leq^* x^r$.
\end{enumerate}
\end{proposition}

\begin{proof}
Let $r_n=s+\frac{1}{n}$, $p_n=$ least $p>p_{n-1}$ such that $\sum_{q=p}^\infty q(q^{-2/s})^{r_n}\leq 2^{-n}$, and
$$f(x)=x^{r_n},\ x\in [p_{n+1}^{-2/s},p_n^{-2/s}).$$
\end{proof}

\begin{corollary}
\label{separate}
For $0<s<1$, there is a gauge function $f$ such that $H^f(W(2/s))=0$ and $H^f(\mathcal{D}_{\leq s})>0$.
\end{corollary}

This gives us a separation of $W(2/s)$ and $\mathcal{D}_{\leq s}$ in Hausdorff measure, which cannot be done by Hausdorff dimension.

\bibliographystyle{plain}
\bibliography{ref}

@article{beresnevich2006mass,
  title={A mass transference principle and the Duffin-Schaeffer conjecture for Hausdorff measures},
  author={Beresnevich, Victor and Velani, Sanju},
  journal={Annals of mathematics},
  pages={971--992},
  year={2006},
  publisher={JSTOR}
}

@article{lutz2003dimensions,
  title={The dimensions of individual strings and sequences},
  author={Lutz, Jack H},
  journal={Information and Computation},
  volume={187},
  number={1},
  pages={49--79},
  year={2003},
  publisher={Elsevier}
}

@inproceedings{ya1984coding,
  title={Coding of combinatorial sources and Hausdorff dimension},
  author={Ya, Ryabko B},
  booktitle={Soviet Math. Doklady},
  volume={30},
  pages={219--222},
  year={1984}
}

@inproceedings{cai1994hausdorff,
  title={On Hausdorff and topological dimensions of the Kolmogorov complexity of the real line},
  author={Cai, Jin-Yi and Hartmanis, Juris},
  booktitle={Proceedings of the 30th IEEE symposium on Foundations of computer science},
  pages={605--619},
  year={1994}
}

@article{lutz2005effective,
  title={Effective fractal dimensions},
  author={Lutz, Jack H},
  journal={Mathematical Logic Quarterly},
  volume={51},
  number={1},
  pages={62--72},
  year={2005},
  publisher={Wiley Online Library}
}

@book{alma991043247829706532,
author = {Rogers, C. A.},
address = {Cambridge [England]},
booktitle = {Hausdorff measures},
isbn = {0521624916},
lccn = {98007167},
publisher = {Cambridge University Press},
language = {eng},
series = {Cambridge mathematical library},
title = {Hausdorff measures},
year = {1998 - 1970},
}

@article{staiger1993kolmogorov,
  title={Kolmogorov complexity and Hausdorff dimension},
  author={Staiger, Ludwig},
  journal={Information and Computation},
  volume={103},
  number={2},
  pages={159--194},
  year={1993},
  publisher={Elsevier}
}

@phdthesis{reimann2004computability,
  title={Computability and fractal dimension},
  author={Reimann, Jan},
  school={Universit\"{a}t Heidelberg},
  year={2004}
}

@article{jarnik1999diophantische,
  title={Diophantische approximationen und hausdorffsches mass},
  author={Jarn{\'\i}k, Vojt{\v{e}}ch},
  journal={Mat. Sb.},
  volune={36},
  pages={371-382},
  year={1929},
}

@article{besicovitch1934sets,
  title={Sets of Fractional Dimensions (IV): On Rational Approximation to Real Numbers},
  author={Besicovitch, AS},
  journal={Journal of the London Mathematical Society},
  volume={1},
  number={2},
  pages={126--131},
  year={1934},
  publisher={Wiley Online Library}
}

@article{jarnik1931simultanen,
  title={{\"U}ber die simultanen diophantischen Approximationen},
  author={Jarn{\'\i}k, Vojt{\v{e}}ch},
  journal={Mathematische Zeitschrift},
  volume={33},
  number={1},
  pages={505--543},
  year={1931},
  publisher={Springer}
}

@article{becher2018irrationality,
  title={Irrationality exponent, Hausdorff dimension and effectivization},
  author={Becher, Ver{\'o}nica and Reimann, Jan and Slaman, Theodore A},
  journal={Monatshefte f{\"u}r Mathematik},
  volume={185},
  number={2},
  pages={167--188},
  year={2018},
  publisher={Springer}
}

@article{staiger2017exact,
  title={Exact constructive and computable dimensions},
  author={Staiger, Ludwig},
  journal={Theory of Computing Systems},
  volume={61},
  number={4},
  pages={1288--1314},
  year={2017},
  publisher={Springer}
}

@article{staiger2020incomputability,
  title={On the incomputability of computable dimension},
  author={Staiger, Ludwig},
  journal={Logical Methods in Computer Science},
  volume={16},
  year={2020},
  publisher={Episciences. org}
}

@incollection{shen2015around,
  title={Around Kolmogorov complexity: basic notions and results},
  author={Shen, Alexander},
  booktitle={Measures of Complexity: Festschrift for Alexey Chervonenkis},
  pages={75--115},
  year={2015},
  publisher={Springer}
}

@article{calude2013liouville,
  title={Liouville, computable, Borel normal and Martin-L{\"o}f random numbers},
  author={Calude, Cristian S and Staiger, Ludwig},
  journal={Theory of Computing Systems},
  volume={62},
  number={7},
  pages={1573--1585},
  year={2018},
  publisher={Springer}
}

@book{beresnevich2006measure,
  title={Measure theoretic laws for lim sup sets},
  author={Beresnevich, Victor and Dickinson, Detta and Velani, Sanju},
  volume={179},
  year={2006},
  publisher={American Mathematical Soc.}
}

@inproceedings{lutz2000gales,
  title={Gales and the constructive dimension of individual sequences},
  author={Lutz, Jack H},
  booktitle={International Colloquium on Automata, Languages, and Programming},
  pages={902--913},
  year={2000},
  organization={Springer}
}

@inproceedings{reimann2005effective,
  title={Effective hausdorff dimension},
  author={Reimann, Jan and Stephan, Frank},
  booktitle={Logic Colloquium},
  volume={1},
  pages={369--385},
  year={2005}
}

@book{downey2010algorithmic,
  title={Algorithmic randomness and complexity},
  author={Downey, Rodney G and Hirschfeldt, Denis R},
  year={2010},
  publisher={Springer Science \& Business Media}
}
\end{document}